\documentclass[reqno]{amsart}
\makeatletter\def\@biblabel#1{#1.}\makeatother
\usepackage{amsmath,amsfonts,amssymb,amsthm,mathrsfs,bbm,stmaryrd,array,titlesec,subfigure}
\usepackage[colorlinks, citecolor=blue]{hyperref}
\newtheorem*{rem}{Remark}
\newtheorem{lemma}{Lemma}
\newtheorem{conjecture}{Conjecture}

\newtheorem{theorem}{Theorem}

\newtheorem{corollary}{Corollary}

\def\mydash{\CJKglue\raise0.2ex\hbox{---\kern-0.01em---}\CJKglue}

\def\ZZ{\mathbb{Z}}

\begin{document}
\title{A New Generalization of Fermat's Last Theorem}
\author{Tianxin Cai, Deyi Chen and Yong Zhang}
\date{}
\maketitle
\vspace{-3em}
%------------------0.Abstract----------------------------------------
\begin{abstract}
In this paper, we consider some hybrid Diophantine equations of addition and multiplication. We first improve a result on new Hilbert-Waring problem. Then we consider the equation
\begin{equation}
  \begin{cases}
      A+B=C\\
      ABC=D^n\\
    \end{cases}
\end{equation}
where $A,B,C,D,n \in\ZZ_{+}$ and $n\geq3$, which may be regarded as a generalization of Fermat's equation $x^n+y^n=z^n$. When $\gcd(A,B,C)=1$, $(1)$ is equivalent to Fermat's equation, which means it has no positive integer solutions. We discuss several cases for $\gcd(A,B,C)=p^k$ where $p$ is an odd prime. In particular, for $k=1$ we prove that $(1)$ has no nonzero integer solutions when $n=3$ and  we conjecture that it is also true for any prime $n>3$. Finally, we consider equation $(1)$ in quadratic fields $\mathbb{Q}(\sqrt{t})$ for $n=3$.
\end{abstract}
\footnote[0]{

Project supported by the National Natural Science Foundation of China 11351002.}
%------------------1.Introduction----------------------------------------

\bigskip
\textbf{1. Introduction}
\smallskip

 In this paper, we consider some hybrid Diophantine equations of addition and multiplication. First of all,
 \begin{equation*}
   n=x_1+x_2+\cdots+x_s
 \end{equation*}
 such that
 \begin{equation*}
   x_1x_2\cdots x_s=x^k,
 \end{equation*}
 for $n,x_{i},x,k \in \ZZ_{+}$, which is a new variant of Waring's problem:
\begin{equation*}
  n=x_1^k+x_2^k+\cdots+x_s^k.
\end{equation*}
 We denote by $g'(k)$ (resp. $G'(k)$) the least positive integer such that every integer (resp. all sufficiently large integer) can be represented as a sum of at most $g'(k)$ (resp. $G'(k)$) positive integers, and the product of the $g'(k)$ (resp. $G'(k)$) integers is a $k$-th power. We show \cite{1} that
 \begin{equation*}
   g'(k)=2k-1;~~G'(p)\leq p+1;~~G'(2p)\leq2p+2~~(p\geq3);~~G'(4p)\leq4p+2~~(p\geq7);
 \end{equation*}
where $k$ is a positive integer and $p$ is prime. In this paper, we improve the results on composite numbers as follow.
\begin{theorem}
  For any composite number  $k$, $G'(k)\leq k+2$.
\end{theorem}

Next, we consider Fermat's Last Theorem. In 1637, Fermat claimed that the Diophantine equation
\begin{equation*}
  x^n+y^n=z^n
\end{equation*}
has no positive integer solutions for any integer $n\geq3$. This was proved finally by Andrew Wiles in 1995 \cite{15,14}.

There are several generalizations of Fermat's Last Theorem, e.g., Fermat-Catalan conjecture, which states that the equation $a^m+b^n=c^k$ has only
finitely many solutions $(a,b,c,m,n,k)$ , where $a,b,c$ are positive coprime integers and $m,n,k$ are positive integers,
satisfying $\frac{1}{m}+\frac{1}{n}+\frac{1}{k}<1$. So far there are only 10 solutions found \cite{4,10}. Meanwhile,  Beal's conjecture \cite{7} states that the equation $A^x+B^y=C^z$ has no solution in positive integers $A,B,C,x,y$ and $z$ with $x,y$ and $z$ at least $3$ and $A,B$ and $C$ coprime. Beal has offered a prize of one million dollars for a proof of his conjecture or a counterexample \cite{16}. Obviously, there are only finite solutions for Beal's equation under Fermat-Catalan conjecture. Meanwhile, it's known that both FLT and Fermat-Catalan conjecture are the consequences of the abc-conjecture, the latter was claimed to be proved in 2012 but not confirmed yet by Japanese mathematician Shinichi Mochizuki \cite{9}.

Now we expand the same idea to Fermat's equation as  we did before to Hilbert-Waring problem. We consider a new Diophantine equation
\begin{equation}
  \begin{cases}
      A+B=C\\
      ABC=D^n\\
    \end{cases}
    \tag{1}
\end{equation}
where $A,B,C,D,n \in\ZZ_{+}$ and $n\geq3$. \\
It is easy to see that if $\gcd(A,B,C)=1$ then $A,B,C$ are  pairwise coprime. Therefore,
\begin{equation}
  \fbox{Equation~$(1)$~has~no~positive~integer solutions~for~$\gcd(A,B,C)=1$  $\Longleftrightarrow$ FLT}
\end{equation}
 In view of (2), we may ask some natural questions concerning $(1)$ :\\
 1. Is it possible to have a solution for all $n\geq3$? If so, is it possible to have infinitely many solutions ? \\
 2. Is it possible to have a solution when $\gcd(A,B,C)=p^k$, where $p$ is a prime, $k\in \ZZ_{+}$ ?\\
 In this paper, we answer the first question affirmatively by proving the following result.
 \begin{theorem}
   For any $n \not\equiv 0 \pmod3, n\geq3 $, $(1)$ has infinitely many positive integer solutions; for any $n \equiv 0~\pmod3, n\geq3 $, $(1)$ has no positive integer solutions.
 \end{theorem}
 For the second question, we discuss the special cases $n=4,5$ and  obtain the following
 \begin{theorem}
   If $\gcd(A,B,C)=p^k$ where $k\in \ZZ_{+}$, p is odd prime and $p \equiv 3~\pmod8$, then the equation
\begin{equation}
   \begin{cases}
      A+B=C\\
      ABC=D^4\\
    \end{cases}
\end{equation}
has no positive integer solutions.
 \end{theorem}

For $p=2$ and some $p \equiv 1,5,7~\pmod8$, it is possible for $(3)$ to have a positive integer solution when $\gcd(A,B,C)=p$. For example
\begin{equation*}
    \begin{cases}
      2+2=4\\
      2\times2\times4=2^4,\\
    \end{cases}
    \begin{cases}
      17+272=289\\
      17\times272\times289=34^4,\\
    \end{cases}
    \begin{cases}
      5+400=405\\
      5\times400\times405=30^4,\\
    \end{cases}
  \end{equation*}
  \begin{equation*}
    \begin{cases}
    47927607119+1631432881=49559040000\\
      47927607119\times1631432881\times49559040000=44367960^4,\\
    \end{cases}
  \end{equation*}
where $p=\gcd(A,B,C)=2, 17,5,239$, so $p \equiv 1,5,7~\pmod8 $ respectively.
 \begin{theorem}
   If $\gcd(A,B,C)=p^k$ where $k\in \ZZ_{+}$, p is odd prime and $p \not\equiv 1~\pmod{10} $, then the equation
 \begin{equation}
  \begin{cases}
      A+B=C\\
      ABC=D^5\\
    \end{cases}
\end{equation}
has no positive integer solutions.
 \end{theorem}
In general, we have the follows:
\begin{conjecture}
  If $n\geq 3$ is prime, $\gcd(A,B,C)=p^k$ where $k\in \ZZ_{+}$, p is odd prime and $p \not\equiv 1~\pmod{2n} $, then $(1)$ has no positive integer solutions.
\end{conjecture}
Finally, if $n>3$ is prime, we construct special prime $p$ such that $(1)$ has  positive integer solutions for $\gcd(A,B,C)=p^k$ as following
\begin{theorem}
  If $n>3$ is prime, $n\equiv r~\pmod3 $, $1\leq r\leq 2$, $a,b,m \in \ZZ$, $m\neq0$, such that $\frac{a^n+b^n}{a+b}=p$ is an odd prime and $a+b=m^n$, then $p\equiv1~\pmod{2n}$ and  $(1)$ has positive integer solutions for $\gcd(A,B,C)=p^k$ where $k \equiv \frac{rn-1}{3}~\pmod{n}$, $k\in \ZZ_{+}$.
\end{theorem}
Let $a=2, b=-1, m=1$, we obtain
\begin{corollary}
  If  $n>3$ is prime, $n\equiv r~\pmod3 $, $1\leq r\leq 2$ and $p=2^n-1$ is Mersenne prime, then  $(1)$ has positive integer solutions for $\gcd(A,B,C)=p^k$ where $k \equiv \frac{rn-1}{3}~\pmod{n}$, $k\in \ZZ_{+}$.
\end{corollary}
Moreover, we have the following
\begin{conjecture}
 If  $n>3$ is prime, $n\equiv r~\pmod3$, $1\leq r\leq 2$ and $\gcd(A,B,C)=p^k$ where $p$ is prime and $k \not\equiv \frac{rn-1}{3}~\pmod{n}$ , then $(1)$ has no positive integer solutions.
\end{conjecture}
In particular, if  $n>3$ is prime, then $4\not\equiv rn~\pmod{n}$, so $1 \not\equiv \frac{rn-1}{3}~\pmod{n}$, it follows that we have a special case of Conjecture $2$ when $k=1$:
\begin{conjecture}
  If n is odd prime, $\gcd(A,B,C)=p$ is prime, then $(1)$ has no positive integer solutions.
\end{conjecture}

\begin{rem}
    If $n=3$, Conjecture 3 is true by Theorem 2. If abc-conjecture is true, then Conjecture 3 should be true for fixed prime $p$ and sufficiently large $n$ where $n$ need not be prime.
\end{rem}
\begin{proof}
Because $\gcd(A,B,C)=p$, then $\gcd(\frac{A}{p},\frac{B}{p},\frac{C}{p})=1$. By $(1)$ we have
  \begin{equation*}
  \begin{cases}
      \frac{A}{p}+\frac{B}{p}=\frac{C}{p}\\
      ABC=D^n.\\
    \end{cases}
\end{equation*}
So $rad\left(\frac{ABC}{p^3}\right)= rad\left(\frac{D}{p^3}\right)\leq rad(D)$ and $C>D^{\frac{n}{3}}$. For any $n\geq7$ and $0<\epsilon<\frac{1}{3}$, we deduce
\begin{equation*}
  p\leq D=D^{\frac{7}{3}-1-\frac{1}{3}}\leq D^{\frac{7}{3}-1-\epsilon}
\end{equation*}
and
\begin{equation*}
  q\left(-\frac{A}{p},-\frac{B}{p},\frac{C}{p}\right)=\frac{\log{(\frac{C}{p})}}{\log{\left(rad(\frac{ABC}{p^3})\right)}}
  \geq\frac{\frac{n}{3}\log{D}-(\frac{7}{3}-1-\epsilon)\log{D}}{\log{D}}\geq1+\epsilon.
\end{equation*}
By abc-conjecture, there exist only finitely many triples $\left(-\frac{A}{p},-\frac{B}{p},\frac{C}{p}\right)$.
Let $A_1=\frac{A}{p},B_1=\frac{B}{p}, C_1=\frac{C}{p}$, then $\gcd(A_1,B_1,C_1)=1$, $p^3A_1B_1C_1=D^n$. Let $M$ be the greatest $m$ such that there is prime $q$ satisfying $q^m|A_1B_1C_1$. Then, if $n>M+3$, there is no solution for $p^3A_1B_1C_1=D^n$. Hence, when $\gcd(A,B,C)=p$, $(1)$ has no positive integer solutions for sufficiently large $n$.
   \end{proof}
However, we could not deduce Conjectures 1-3 from abc-conjecture.

% \begin{conjecture}
%  If $p \not\equiv 1~(mod~2n)$, $gcd(A,B,C)=p^k$, $k\in \ZZ_{+}$, equation
% \begin{equation*}
%  \begin{cases}
%      A+B+C=0\\
%      ABC=D^{n}\\
%    \end{cases}
%\end{equation*}
%has no nonzero solution.
% \end{conjecture}

%-------------------2.preliminaries---------------------------------------
\bigskip

\textbf{2. Preliminaries}
%%--------------------2.1 Lemma1---------------------------------------------
\begin{lemma}\cite[Proposition 6.5.6]{2}
  Let $c$ be a nonzero integer. The equation $x^4-y^4=cz^2$ has a rational solution with $xyz\neq0$ if and only if $|c|$ is a congruent number.
\end{lemma}
\begin{lemma}\cite[Proposition 5]{13}
  Let $p$ be a prime congruent to 3 modulo 8, then p is not a congruent number.
\end{lemma}
\begin{lemma}\cite{5}
  Let $A>2$ be a positive integer and has no prime divisors of the form $10k+1$, then the equation
  \begin{equation}
    \begin{cases}
      x^5+y^5=Az^5\\
      \gcd(x,y)=1\\
    \end{cases}
  \end{equation}
   has no nonzero integer solution. If $A=2$, the solutions of $(5)$ are $(x,y,z)=\pm(1,1,1)$.
\end{lemma}
This lemma was first conjectured by V. A. Lebesgue \cite{6} in 1843 and  proved by E. Halberstadt and A. Kraus \cite{5} in 2004.
\begin{lemma}
  For any prime $p$, integer $n\geq2$, if $\gcd(A,B,C)=p^k$ and $k \equiv 0~\pmod{n}$, $k\in \ZZ_{+}$, then $(1)$ has no nonzero integer solutions.
\end{lemma}
\begin{proof}
  Let $A_{1}=\frac{A}{p^k}$, $B_{1}=\frac{B}{p^k}$, $C_{1}=\frac{C}{p^k}$, in view of $\gcd(A,B,C)=p^k$ and $(1)$, we obtain that $A_{1},B_{1},C_{1}$ are pairwise coprime and $(1)$ can be changed into
  \begin{equation*}
    \begin{cases}
      A_{1}+B_{1}=C_{1}\\
      p^{3k}A_1B_1C_1=D^n.
    \end{cases}
  \end{equation*}
  But $k \equiv 0~\pmod{n}$, $A_{1},B_{1},C_{1}$ are pairwise coprime, so $A_{1}=x^n$, $B_{1}=y^n$, $C_{1}=z^n$ and $x^n+y^n=z^n$. By Fermat's Last Theorem, we deduce that $xyz=0$, so $ABC=0$. Contradiction.
\end{proof}
%--------------------3.The proofs of the theorems---------------------------------------
\bigskip

\textbf{3. Proofs of the Theorems}
%%--------------------3.1 Theorem 1---------------------------------------------
\begin{proof}[\textbf{Proof of Theorem 1} \\]
For every positive integer $n$, let $n=km+r$ where $0\leq r\leq k-1$.\\
If $r=0$, when $n>2k^{2k}$, we have $m=\frac{n}{k}>2k^{2k-1}$ and
\begin{equation*}
  n=km=\underbrace{(m-2k^{2k-1})+\cdots+(m-2k^{2k-1})}_{k}+k^{2k}+k^{2k}.
\end{equation*}
If $0<r \leq k-1$, when $n>k^{2k-1}+k$, we have $m=\frac{n-r}{k}>\frac{k^{2k-1}}{k}>k^{k-1}r^{k-1}$ and
\begin{equation*}
  n=km+r=\underbrace{(m-k^{k-1}r^{k-1})+\cdots+(m-k^{k-1}r^{k-1})}_{k}+k^{k}r^{k-1}+r.
\end{equation*}
\end{proof}
%%--------------------3.2 Theorem 2---------------------------------------------
\begin{proof}[\textbf{Proof of Theorem 2} \\]
If $n \not\equiv 0~\pmod3 $, then there exists $k\in\ZZ_{+}$ such that $3k+2 \equiv 0~\pmod{n}$. It is well known that there exist infinitely many $(a,b,c)\in\ZZ_{+}^{3}$ such that $a^2+b^2=c^2$. Let
\begin{equation*}
  \begin{cases}
    A=a^{k+2}b^kc^k\\
    B=a^kb^{k+2}c^k\\
    C=a^kb^kc^{k+2}.
  \end{cases}
\end{equation*}
So we have infinitely many positive solutions $(A,B,C)$ satisfying $(1)$, where $D=(abc)^{\frac{3k+2}{n}}$. \\
If $n \equiv 0~\pmod{3}$, suppose $(A,B,C)$ is a positive solution of $(1)$. Let $d=\gcd(A,B,C)$. Then $\frac{A}{d},\frac{B}{d}$ and $\frac{C}{d}$ are pairwise coprime and $d^3|D^n$. But $3|n$, so we have $d|D^{\frac{n}{3}}$ and
\begin{equation*}
  \begin{cases}
    \frac{A}{d}+\frac{B}{d}=\frac{C}{d}\\
    \frac{A}{d}\cdot\frac{B}{d}\cdot\frac{C}{d}=(\frac{D^{\frac{n}{3}}}{d})^3.\\
  \end{cases}
\end{equation*}
Since $\frac{A}{d},\frac{B}{d}$ and $\frac{C}{d}$ are pairwise coprime, so $\frac{A}{d}=x^3,\frac{B}{d}=y^3,\frac{C}{d}=z^3$ and $x^3+y^3=z^3$. By Fermat's Last Theorem, we deduce that $xyz=0$ and  $ABC=0$. Contradiction.
\end{proof}
%%--------------------3.3 Theorem 3---------------------------------------------
\begin{proof}[\textbf{Proof of Theorem 3} \\]
In view of Lemma 4, we only need to discuss $k\not\equiv0~\pmod4 $. Suppose odd prime $p \equiv 3~\pmod8$ and $(3)$ has a positive integer solution $(A,B,C)$. In view of $\gcd(A,B,C)=p^k$, $(3)$ can be changed into
  \begin{equation}
    \begin{cases}
      A_{1}+B_{1}=C_{1}\\
      p^{3k}A_1B_1C_1=D^4.
    \end{cases}
  \end{equation}
where $A_{1}=\frac{A}{p^k}$, $B_{1}=\frac{B}{p^k}$ and  $C_{1}=\frac{C}{p^k}$ are pairwise coprime. But $k\not\equiv0~\pmod4$, we deduce that only one of $A_1,B_1,C_1$ is congruent to $0$ modulo $p$. Let $3k\equiv r~\pmod4$, then $1\leq r\leq 3$. In view of (6) and reorder $A_1,B_1$ if necessary, we obtain
\begin{equation}
  \begin{cases}
      A_1=x^4\\
      B_1=y^4,\\
      C_1=p^{4-r}z^4,\\
    \end{cases}
\end{equation}
or
\begin{equation}
    \begin{cases}
      A_1=p^{4-r}z^4,\\
      B_1=y^4,\\
      C_1=x^4,\\
    \end{cases}
  \end{equation}
where $x,y,pz$ are pairwise coprime. \\
If $A_1, B_1$ and $C_1$ satisfy $(7)$, then
\begin{equation}
  x^4+y^4=p^{4-r}z^4.
\end{equation}
But $\gcd(y,p)=1$, we have integer $s\not\equiv0~\pmod{p}$ such that $sy\equiv1~\pmod{p}$. By $(9)$ we deduce that
\begin{equation*}
  (xs)^4\equiv~-1~\pmod{p},
\end{equation*}
which implies that  $-1$ is a square modulo $p$,  this can only hold for $p=2$  and $p \equiv 1 \pmod{4}$, in contradiction with $p\equiv 3~\pmod8$.\\
If $A_1, B_1$ and $C_1$ satisfy $(8)$, then
\begin{equation}
  x^4-y^4=p^{4-r}z^4.
\end{equation}
When $r=2$, we have $x^4-y^4=(pz^2)^2$, but it is well known that the equation $X^4-Y^4=Z^2$ has no nonzero integer solutions. So $r=1$ or $r=3$ and  $(10)$ can be changed into
\begin{equation}
  x^4-y^4=p(pz^2)^2
\end{equation}
or
\begin{equation}
  x^4-y^4=p(z^2)^2
\end{equation}
respectively. By Lemma 1 and Lemma 2, we obtain that $(11)$ and $(12)$ have no positive integer solutions when $p\equiv 3~\pmod8$.
\end{proof}
%%--------------------3.4 Theorem 4---------------------------------------------
\begin{proof}[\textbf{Proof of Theorem 4} \\]
In view of Lemma 4, we only need to discuss $k\not\equiv0~\pmod5 $. In view of $\gcd(A,B,C)=p^k$, $(4)$ can be changed into
  \begin{equation}
    \begin{cases}
      A_{1}+B_{1}=C_{1}\\
      p^{3k}A_1B_1C_1=D^5.
    \end{cases}
  \end{equation}
where $A_{1}=\frac{A}{p^k}$, $B_{1}=\frac{B}{p^k}$ and $C_{1}=\frac{C}{p^k}$ are pairwise coprime. But $k\not\equiv0~\pmod5$, we deduce that only one of $A_1,B_1,C_1$ congruent to $0$ modulo $p$. Let $3k\equiv r~\pmod5$, then $1\leq r\leq 4$. In view of (13) and reorder $A_1,B_1,C_1$ if necessary, we obtain
\begin{equation}
  \begin{cases}
      A_1=x^5\\
      B_1=y^5,\\
      C_1=p^{5-r}z^5,\\
    \end{cases}
\end{equation}
where $x,y$ and $pz$ are pairwise coprime.
From $(13)$ and $(14)$ we deduce that
\begin{equation}
  x^5+y^5=p^{5-r}z^5.
\end{equation}
But $p$ is an odd prime and $p \not\equiv 1~\pmod{10}$, by Lemma 3, $(15)$ has no positive integer solutions.
\end{proof}
%%--------------------3.5 Theorem 5---------------------------------------------
\begin{proof}[\textbf{Proof of Theorem 5} \\]
First of all, we prove that $f(x,y)=\frac{x^k+y^k}{x+y}>0$ for any odd positive integer $k$, where $(x,y) \in \mathbb{R}^2$, $x+y\neq0$. It is obvious that $f(x,y)>0$ when $xy\geq0$. So we only need to prove the case $xy<0$. If $k=1$, $f(x,y)=\frac{x+y}{x+y}=1>0$. Suppose $\frac{x^k+y^k}{x+y}>0$ where $x+y\neq0$. Because $xy<0$ and $x+y\neq0$, we obtain
\begin{equation*}
  \frac{x^{k+2}+y^{k+2}}{x+y}=\frac{x^k+y^k}{x+y}(-xy)+x^{k+1}+y^{k+1}>0.
\end{equation*}
By mathematical induction, we deduce that $f(x,y)=\frac{x^k+y^k}{x+y}>0$ for any odd positive integer $k$, where $(x,y) \in \mathbb{R}^2$, $x+y\neq0$.
By condition of Theorem 5, we have a prime $p$ such that
\begin{equation}
  0<p=\frac{a^n+b^n}{a+b}=a^{n-1}-a^{n-2}b+a^{n-3}b^2-\cdots+b^{n-1}=\sum_{j=0}^{n-1}a^{n-1-j}(-b)^j.
\end{equation}
Since $n\geq3$, we deduce from $(16)$ that
\begin{equation}
  \gcd(p,a)=\gcd(p,b)=\gcd(a,b)=1.
\end{equation}
So
\begin{equation}
  \gcd(a,a+b)=\gcd(b,a+b)=1.
\end{equation}
Let
\begin{equation}
  \begin{cases}
    A=p^{\frac{rn-1}{3}+tn}a^n\\
    B=p^{\frac{rn-1}{3}+tn}b^n\\
    C=p^{\frac{rn+2}{3}+tn}(a+b)
  \end{cases}
\end{equation}
where $r$ is defined in the condition of Theorem 5 and integer $t\geq0$.\\
By $(17)$, $(18)$, $(19)$ and $a+b=m^n$, we obtain that $A,B$ and $C$ satisfy $(1)$ where $D=p^{r+3t}abm$ and
\begin{equation*}
  \gcd(A,B,C)=p^{k},
\end{equation*}
where $k=\frac{rn-1}{3}+tn\equiv\frac{rn-1}{3}~\pmod{n}$.\\
Finally, we want to prove that $p\equiv1~\pmod{2n}$. We only need to prove that $p\equiv1~\pmod{n}$ because both $p$ and $n$ are odd primes.\\
Since $p=\frac{a^n+b^n}{a+b}$, $n\geq3$ is prime, by Fermat's little Theorem, we have
\begin{equation*}
  p(a+b)=a^n+b^n\equiv a+b~\pmod{n}.
\end{equation*}
Now we only need to prove that $a+b\not\equiv0~\pmod{n}$. Suppose $a+b\equiv0~\pmod{n}$, then $-b\equiv a~\pmod{n}$ and
\begin{equation*}
  p=\frac{a^n+b^n}{a+b}=a^{n-1}-a^{n-2}b+a^{n-3}b^2-\cdots+b^{n-1}\equiv na^{n-1}\equiv0~\pmod{n}.
\end{equation*}
But $p$ and $n$ are two primes, so we deduce that $p=n$.\\
If $ab<0$, We may assume that $a>0$, $b<0$, then
\begin{equation*}
  p=\frac{a^n+b^n}{a+b}=a^{n-1}-a^{n-2}b+a^{n-3}b^2-\cdots+b^{n-1}>\underbrace{1+1+\cdots+1}_{n}=n=p.
\end{equation*}
contradiction. \\
If $ab\geq0$, we may assume that $a\geq0$, $b\geq0$, but $p=\frac{a^n+b^n}{a+b}$ is prime, so $a\neq b$ and $ab\neq0$. Reorder $a,b$ (if necessary), we may assume that $a\geq1$, $b\geq2$.\\
If $a=1$,$b\geq2$, then $p=\frac{1+b^n}{1+b}=\frac{1+b^p}{1+b}$. Let $f(b)=(b^p+1)-p(b+1)$, then
\begin{equation*}
  f'(b)=pb^{p-1}-p=p(b^{p-1}-1)>0,
\end{equation*}
so
\begin{equation*}
  f(b)\geq f(2)=(2^p+1)-3p>0
\end{equation*}
for odd prime $p$, from which we deduce that $\frac{1+b^p}{1+b}>p$. Contradiction.\\
If $a\geq2$, $b\geq2$, then $p=\frac{a^n+b^n}{a+b}=\frac{a^p+b^p}{a+b}>\frac{pa+pb}{a+b}=p$, once again a contradiction!\\
\end{proof}
%%--------------------3.8 Table---------------------------------------------
Finally, we list some primes $p$ satisfying the condition of Theorem 5 when $n=5,7$.
\begin{table}[htbp]
\caption{Primes $p<10^7$ which satisfy the condition of Theorem 5 when $n=5$.}
\centering
\subtable{
\begin{tabular}{cccc}
\hline
$p$ & $a$ & $b$ & $m$ \\
\hline
$31$ & $2$ & $-1$ & $1$ \\
\hline
$211$ & $3$ & $-2$ & $1$ \\
\hline
$4651$ & $6$ & $-5$ & $1$ \\
\hline
$61051$ & $11$ & $-10$ & $1$ \\
\hline
$132661$ & $11$ & $21$ & $2$ \\
\hline
$202981$ & $9$ & $23$ & $2$ \\
\hline
$371281$ & $17$ & $-16$ & $1$ \\
\hline
$723901$ & $20$ & $-19$ & $1$ \\
\hline
\end{tabular}
}
\qquad
\subtable{
\begin{tabular}{cccc}
\hline
$1641301$ & $35$ & $-3$ & $2$ \\
\hline
$1803001$ & $25$ & $-24$ & $1$ \\
\hline
$2861461$ & $28$ & $-27$ & $1$ \\
\hline
$4329151$ & $31$ & $-30$ & $1$ \\
\hline
$4925281$ & $32$ & $-31$ & $1$ \\
\hline
$5754901$ & $45$ & $-13$ & $2$ \\
\hline
$7086451$ & $35$ & $-34$ & $1$ \\
\hline
$7944301$ & $36$ & $-35$ & $1$ \\
\hline
$8782981$ & $49$ & $-17$ & $2$ \\
\hline
\end{tabular}
}
\end{table}\\
\begin{table}[htbp]
\caption{Primes $p<10^{11}$ which satisfy the condition of Theorem 5 when $n=7$.}
\centering
\subtable{
\begin{tabular}{cccc}
\hline
$p$ & $a$ & $b$ & $m$ \\
\hline
$127$ & $2$ & $-1$ & $1$ \\
\hline
$14197$ & $4$ & $-3$ & $1$ \\
\hline
$543607$ & $7$ & $-6$ & $1$ \\
\hline
$1273609$ & $8$ & $-7$ & $1$ \\
\hline
$2685817$ & $9$ & $-8$ & $1$ \\
\hline
$5217031$ & $10$ & $-9$ & $1$ \\
\hline
$16344637$ & $12$ & $-11$ & $1$ \\
\hline
$141903217$ & $17$ & $-16$ & $1$ \\
\hline
\end{tabular}
}
\qquad
\subtable{
\begin{tabular}{cccc}
\hline
$1928294551$ & $26$ & $-25$ & $1$ \\
\hline
$8258704609$ & $33$ & $-32$ & $1$ \\
\hline
$14024867221$ & $36$ & $-35$ & $1$ \\
\hline
$22815424087$ & $39$ & $-38$ & $1$ \\
\hline
$30914273881$ & $41$ & $-40$ & $1$ \\
\hline
$77617224511$ & $59$ & $69$ & $2$ \\
\hline
$91154730577$ & $49$ & $-48$ & $1$ \\
\hline
$98201826199$ & $55$ & $73$ & $2$ \\
\hline
\end{tabular}
}
\end{table}\\

\bigskip

\newpage

\textbf{4. New Fermat equation in quadratic fields}
\smallskip

It's well-known that Fermat's equation $x^n+y^n=z^n$ has only
the trivial solutions in integers when $n\geq 3$. Therefore, it is
an interest problem that whether Fermat's equation has
non-trivial solutions in algebraic number fields. There are numerus
papers on this problem and we can refer to \cite{20,19} and the
references there. For the case $n=3$, it was solved almost
completely. In 1915, W. Burnside \cite{20} proved that in  quadratic
field  Fermat's equation $x^3+y^3=z^3$ has solutions of the form
\[\begin{cases}
\begin{split}
&x=-3+\sqrt{-3(1+4k^3)},\\
&y=-3-\sqrt{-3(1+4k^3)},\\
&z=6k,
\end{split}\end{cases}\]
where $k$ is a rational number not equal to $0$ and $-1$. While $k=0$, Fermat's equation $x^3+y^3=z^3$ has no non-trivial solutions in $\mathbb{Q}\sqrt{-3}$.

In 2013, M. Jones and J. Rouse \cite{19} gave necessary and sufficient conditions
on a square-free integer $t$ such that $x^3+y^3=z^3$ has a nontrivial
solution in  quadratic fields $\mathbb{Q}(\sqrt{t})$, under the
Birch and Swinnerton-Dyer conjecture.

Now we consider the new Fermat equation (1) for $n=3$ in quadratic
fields $\mathbb{Q}(\sqrt{t})$. Assume that the solution in quadratic
fields $\mathbb{Q}(\sqrt{t})$ has the form $a+b\sqrt{t}$ with
$ab\neq 0.$ We have the following theorem.
\begin{theorem}
For any square-free integer $t\neq0,1$ such that the elliptic curve
\[tu^2=1+4k^3\] has a nonzero rational solution $(u,k)$, then $(1)$ has
infinitely many solutions $(A,B,C,D)$ in $\mathbb{Q}(\sqrt{t})$ for
$n=3$.
\end{theorem}

\begin{proof}[\textbf{Proof of Theorem 6}]
Let $A=a+b\sqrt{t},B=c+d\sqrt{t},D=e+f\sqrt{t}$. When $n=3$, we get from $(1)$ that
\[\begin{split}
&a^2c+ac^2+2adtb+ad^2t+b^2tc+2btcd+(2acb+2acd+a^2d\\
&+bc^2+tb^2d+tbd^2)\sqrt{t}=e^3+3ef^2t+f(3e^2+f^2t)\sqrt{t}.\end{split}\]
Then
\[
\begin{cases}\begin{split}
&a^2c+ac^2+2adtb+ad^2t+b^2tc+2btcd=e^3+3ef^2t,\\
&2acb+2acd+a^2d+bc^2+tb^2d+tbd^2=f(3e^2+f^2t).
\end{split}\end{cases}\]
Solving the above two equations, we have
\[t=-\frac{a^2c+ac^2-e^3}{ad^2+b^2c+2bcd+2adb-3ef^2}=-\frac{2acb+2acd+a^2d+bc^2-3fe^2}{b^2d+bd^2-f^3}.\]
Taking $e=kc,f=kd,$ from the formula of $t$, we  have
\[(ad+2ab+cb-2ck^3d)(d^2a^2+c^2b^2+2cbda+2c^2db+2cd^2a-4c^2d^2k^3)=0.\]
Let us consider $ad+2ab+cb-2ck^3d=0,$ then
\[d=-\frac{b(c+2a)}{a-2ck^3}.\]
Hence, we get \[t=\frac{(a-2ck^3)^2}{(1+4k^3)b^2}.\] Put
$a-2ck^3=(1+4k^3)b$, then
\[t=1+4k^3.\]
Therefore, for $n=3$, (1) has solutions
\[
\begin{cases}\begin{split}
&A=(1+4k^3)a+(a-2k^3c)\sqrt{1+4k^3},\\
&B=(1+4k^3)c-(2a+c)\sqrt{1+4k^3},\\
&C=(1+4k^3)(a+c)-(a+(2k^3+1)c)\sqrt{1+4k^3},\\
&D=(1+4k^3)kc-k(2a+c)\sqrt{1+4k^3},
\end{split}\end{cases}\] where $a,c$  and $k$ are non-zero rational numbers.

Let $tu^2=1+4k^3,$ where $t$ is a square-free integer and $t\neq
0,1.$ Then for $n=3$ the new Fermat equation (1) has solutions
\[
\begin{cases}\begin{split}
&A=uta+(a-2k^3c)\sqrt{t},\\
&B=utc-(2a+c)\sqrt{t},\\
&C=ut(a+c)-(a+(2k^3+1)c)\sqrt{t},\\
&D=utkc-k(2a+c)\sqrt{t}.
\end{split}\end{cases}\]
This completes the proof of Theorem 6.
\end{proof}

As an example, taking $k=-1$, then (1) has solutions
\[
\begin{cases}\begin{split}
&A=-3a+(a+2c)\sqrt{-3},\\
&B=-3c-(2a+c)\sqrt{-3},\\
&C=-3(a+c)-(a-c)\sqrt{-3},\\
&D=3c+(2a+c)\sqrt{-3},
\end{split}\end{cases}\]
where $a,c$ are non-zero rational numbers.
\smallskip

To get Burnside's solutions for a given $t$, we need to consider the
following elliptic curve
\[tu^2=-3(1+4k^3).\]
By some calculations, we find that the elliptic curves $tu^2=1+4k^3$ and
$tu^2=-3(1+4k^3)$ have the same $j$-invariant, so they are
isomorphic and have the same rank. If they have the rank greater
than 1, then there are infinitely many rational solutions $(u,k)$
for both of these two elliptic curves. If they have the rank zero,
we can't distinguish the torsion points on them, they might have or not have non-zero rational solutions. For $-50\leq t\leq 50$ and $t$ is square-free, we find no other $t$ as the above example $t=-3$. So we may ask the
following question.

\smallskip

\textbf{Question:} \emph{Are there other $t\neq -3$ such that for $n=3$ the new Fermat
equation $(1)$ has non-trivial solutions but the Fermat equation $x^3+y^3=z^3$ hasn't?}

%--------------------5.Bibliography---------------------------------------

\bigskip
\noindent{Department of Mathematics, Zhejiang University, Hangzhou, 310027, China}\\
\textbf{Email address: txcai$@$zju.edu.cn}
\bigskip

\noindent{Department of Mathematics, Zhejiang University, Hangzhou, 310027, China}\\
\textbf{Email address: chendeyi1986$@$126.com}
\bigskip

\noindent{Department of Mathematics, Zhejiang University, Hangzhou, 310027, China}\\
\textbf{Email address: zhangyongzju@163.com}
\end{document}